\documentclass[10pt]{amsart}
\usepackage[utf8x]{inputenc}
\usepackage{dsfont}
\usepackage{amsmath}
\usepackage{amssymb}
\usepackage{graphicx}
\usepackage{amssymb}
\usepackage{amsfonts}
\usepackage{soul}
\usepackage{mathpazo}
\usepackage{color}
\usepackage{yfonts}
\usepackage{paralist}
\usepackage{stmaryrd}
\usepackage{amsxtra}
\usepackage{latexsym,mathrsfs}
\usepackage{verbatim}
\usepackage{mathabx}

\usepackage{epsfig, enumerate}
\usepackage{color}
\usepackage{hyperref}

\newtheorem{theorem}{Theorem}
\newtheorem*{lemma*}{Claim}
\newtheorem{lemma}[theorem]{Lemma}
\newtheorem{coro}[theorem]{Corollary}

\newtheorem{prop}[theorem]{Proposition}

 \theoremstyle{definition}

\numberwithin{theorem}{section}
\numberwithin{equation}{section}
\definecolor{turquoise}{cmyk}{0.65,0,0.1,0.1}
\definecolor{purple}{rgb}{0.65,0,0.65}
\definecolor{green}{rgb}{0, 0.5, 0}
\definecolor{blue}{rgb}{0, 0, 1}
\definecolor{orange}{rgb}{0.8, 0.6, 0.2}
\definecolor{red}{rgb}{0.8, 0.2, 0.2}
\definecolor{brown}{rgb}{0.5, 0.16, 0.16}

\title[components]{Exact three-component covers in 2-coloured random bipartite graphs}

\author{Xiao-Chuan Liu}
\address[Liu]{Departamento de Matemática,
 Universidade Federal de Pernambuco,
	Avenida Jornalista Aníbal Fernandes, Cidade Universitária, Recife, Brasil}
\email{xiaochuan.liu@ufpe.br}

\author{Xu Yang}
\address[Yang]{Instituto de Computação,  Universidade Federal de Alagoas,
	Av. Lourival Melo Mota, S/N, Maceió, 57072-900, Brasil}
\email{yang@ic.ufal.br}

\begin{document}
\maketitle

\begin{abstract}
We resolve the two-colour three-component conjecture of Fernández,
Pavez-Signé and Stein for random bipartite graphs. More precisely, we prove
that if \(G\sim G(n,n,p)\) and \(p\gg\sqrt{\log n/n}\), then with high
probability every red--blue edge-colouring of \(G\) admits a cover of its
vertex set by at most three monochromatic connected components. The proof is
based on a uniform expansion lemma for unions of common neighbourhoods and an alternating common-neighbourhood expansion argument.
\end{abstract}

\section{Introduction}

In Ramsey-type partition and covering questions, one asks: given a host graph
\(G\) whose edges are coloured in \(r\) colours, how few monochromatic
connected subgraphs suffice to cover, or partition, its vertex set? In dense
deterministic settings, this line connects to classical conjectures such as
those of Erdős--Gyárfás--Pyber and generalizations of Ryser--Lovász, especially
for complete or complete bipartite host graphs.

In the random graph setting, Bal and DeBiasio \cite{BalDeBiasio17}
initiated the probabilistic analogue of the two-colour edge-colouring
problem. They proved that if
\(p\ge (27\log n/n)^{1/3}\), then asymptotically almost surely every
red--blue edge-colouring of \(G(n,p)\) admits a partition of its vertex set
into two monochromatic connected components. On the other hand, they showed
that if \(p\ll (2\log n/n)^{1/2}\), then asymptotically almost surely there
exists a red--blue edge-colouring of \(G(n,p)\) for which no bounded number
of monochromatic components covers all vertices. They conjectured that the
threshold for the two-colour partition property is of order
\((\log n/n)^{1/2}\).

Kohayakawa, Mota and Schacht \cite{KohayakawaMotaSchacht16}
resolved this conjecture in the two-colour case. They proved that
\(p(n)=\sqrt{\log n/n}\) is the threshold for the property that,
asymptotically almost surely, every red--blue edge-colouring of \(G(n,p)\)
admits a partition of \(V(G(n,p))\) into two vertex-disjoint monochromatic
trees.

In the bipartite random graph model, Fernández, Pavez-Signé and Stein
\cite{FernandezPavezStein26} proposed the corresponding exact
three-component conjecture for \(G(n,n,p)\): namely, that
\(p(n)=\sqrt{\log n/n}\) is the threshold for the property that,
asymptotically almost surely, every red--blue edge-colouring of \(G(n,n,p)\)
admits a cover of \(V(G(n,n,p))\) by at most three monochromatic connected
components.

They proved an approximate form of this conjecture: if
\(p\gg \sqrt{\log n/n}\), then asymptotically almost surely every red--blue
edge-colouring of \(G(n,n,p)\) admits a cover of all but \(O(1/p)\) vertices
by at most three vertex-disjoint monochromatic connected components. They
also proved the corresponding lower bound: for some constant \(c>0\), if
\(p\le c\sqrt{\log n/n}\), then asymptotically almost surely
\(tc_2(G(n,n,p))\ge 4\). Our main result removes the exceptional set in the
upper range, thereby establishing the exact three-component cover above the
conjectured threshold.

\begin{theorem}\label{thm:three-component-bipartite}
Let \(p=p(n)\) satisfy \(p \gg \sqrt{\log n/n}\), and let
\(G \sim G(n,n,p)\). Then asymptotically almost surely, every red--blue
edge-colouring of \(G\) admits a cover of its vertex set by at most three
monochromatic connected components.
\end{theorem}

The proof has two main ingredients. First, we establish a uniform expansion
statement for unions of common neighbourhoods in the random bipartite graph.
Second, we use this expansion to route the small exceptional sets left after
applying large-component arguments, thereby obtaining an exact
three-component cover.

Let us also note that a simpler variant of the same method reproves the
corresponding upper-bound cover statement for ordinary random graphs: if
\(p\gg\sqrt{\log n/n}\), then asymptotically almost surely every red--blue
edge-colouring of \(G(n,p)\) admits a cover of its vertex set by two
monochromatic connected components. We omit the details.

The remainder of the paper is organized as follows. In Section~2 we collect
the probabilistic tools used throughout the proof. Section~3 proves the main
expansion and connectivity lemmas. Section~4 completes the proof of
Theorem~\ref{thm:three-component-bipartite}.

\section{Preliminaries}
Throughout the paper, we omit floor and ceiling signs whenever they have no
effect on the argument.
\begin{lemma}\label{lem:complete-graph-two-colour}
Let the edges of \(K_n\) be coloured red and blue. Then \(V(K_n)\)
is contained in a single monochromatic connected component.
\end{lemma}

\begin{proof}
Choose a maximal red component $R$. If $R=V(K_n)$ we are done. Otherwise, each $y\in V\setminus R$ cannot have red edges into $R$, so all its edges to $R$ are blue. Thus every such $y$ lies in the same blue component as $R$, so $V$ is covered by a single blue component.
\end{proof}

\begin{lemma}[Degree and codegree]\label{degree_codegree}
Let \(G\sim G(n,n,p)\) be the random bipartite graph with parts \(A\) and
\(B\), where \(|A|=|B|=n\), and let \(p=p(n)\) satisfy
\(p\gg \sqrt{\log n/n}\). Then, with high probability:
\begin{enumerate}
    \item \textbf{Degrees.} Every vertex \(v\) satisfies
    \[
      \deg(v) \in \big[(1-\tfrac12)pn,\,(1+\tfrac12)pn\big].
    \]
    In particular, \(\deg(v)\ge \tfrac12 pn \gg \sqrt{n\log n}\).

    \item \textbf{Common neighbourhoods.} For every two distinct vertices
    \(u,v\) on the same side of the bipartition,
    \[
      |N(u)\cap N(v)|
      \in \big[(1-\tfrac12)p^2 n,\,(1+\tfrac12)p^2 n\big].
    \]
    In particular,
    \[
      |N(u)\cap N(v)|\ge \tfrac12 p^2 n \gg \log n.
    \]
\end{enumerate}
\end{lemma}

\begin{proof}
For any fixed vertex $v\in A$ (or $B$), $\deg(v)\sim \mathrm{Bin}(n,p)$. By the Chernoff bound,
\[
\mathbb{P}\!\big(\deg(v) < (1-\tfrac12)pn\big) \le \exp(-pn/8).
\]
A union bound over all $2n$ vertices shows that w.h.p.\ no vertex violates this as soon as $pn \gg \log n$. Our assumption $p\gg \sqrt{\log n/n}$ implies $pn \gg \sqrt{n\log n}\gg \log n$, so (1) holds; the upper tail is analogous.

For (2), fix distinct $u,v$ on the same side. Their common neighbourhood size
$|N(u)\cap N(v)|\sim \mathrm{Bin}(n,p^2)$, and Chernoff gives
\[
\mathbb{P}\!\big(|N(u)\cap N(v)| < (1-\tfrac12)p^2 n\big) \le \exp(-p^2 n/8).
\]
There are $2\binom{n}{2}=O(n^2)$ such pairs; by a union bound, w.h.p.\ all pairs satisfy the bound provided $p^2 n \gg \log n^2 \asymp \log n$. This is exactly our regime $p\gg \sqrt{\log n/n}$. The upper tail is handled similarly.
\end{proof}

\begin{lemma}[Uniform \(\log n\)-expansion]\label{logn_expanding}
Let \(p=p(n)\gg \sqrt{\log n/n}\), and let
\(G\sim G(n,n,p)\) be the random bipartite graph with parts \(A\) and \(B\).
Then, with high probability, the following holds simultaneously.

For every integer \(m\) with \(1\le m\le \sqrt{n/\log n}\), and for every
choice of distinct vertices \(x',x_1,\ldots,x_m\) all lying in the same part
of the bipartition, if
\[
U:=\bigcup_{i=1}^m \bigl(N(x')\cap N(x_i)\bigr),
\]
where the neighbourhoods are taken in the opposite part, then
\[
|U|=\omega(m\log n).
\]
\end{lemma}

\begin{proof}
First choose $x',x_1,\dots,x_m\in A$ with $1\le m\le \sqrt{n/\log n}$ fixed, and for each $b\in B$ let
$I_b=\mathbf 1_{\{b\in U\}}$. 
The edge indicators of $\{bx',bx_i \,\big| \, i\in[m]\}$ are independent $\mathrm{Ber}(p)$ random variables;  for distinct $b$ these sets are disjoint, hence
$\{I_b\}_{b\in B}$ are i.i.d. with
\begin{equation} 
\mathbb{P}(I_b=1)=q:=p\big(1-(1-p)^m\big).
\end{equation}
Therefore $|U|=\sum_{b\in B} I_b\sim\mathrm{Bin}(n,q)$ and $\mathbb{E}[\,|U|\,]=nq$.

By assumption, we may write
\[
p=\alpha(n)\sqrt{\frac{\log n}{n}},
\]
where \(\alpha(n)\to\infty\). We consider two regimes.

\smallskip
\emph{Case 1: $pm\le 1$.}
Using $1-(1-p)^m \ge pm-\frac12 p^2m^2$,
\begin{equation}
q \;\ge\; p\Big(pm-\tfrac12 p^2m^2\Big)
\;=\; mp^2\Big(1-\tfrac12 pm\Big)
\;\ge\; \tfrac12\,mp^2,
\end{equation}
hence
\begin{equation}
\mathbb{E}[\,|U|\,] \;\ge\; \tfrac12\,mp^2 n
\;=\; \tfrac12\, m\,\alpha(n)^2\log n
\;=\; \omega(m\log n).
\end{equation}
By Chernoff’s lower tail with $\varepsilon=\tfrac12$,
\begin{equation}
\mathbb{P}\Big(|U|<\tfrac12\,\mathbb{E}[\,|U|\,]\Big)
\;\le\; \exp\!\Big(-\tfrac{\mathbb{E}[\,|U|\,]}{8}\Big)
\;\le\; \exp\!\big(-\Omega(mp^2 n)\big)
\;=\; \exp\!\big(-\Omega(\alpha(n)^2\,m\log n)\big).
\end{equation}

\smallskip
\emph{Case 2: $pm>1$.}
Since $(1-p)^m\le e^{-pm}$,
\begin{equation}
q \;\ge\; p\big(1-e^{-pm}\big) \;\ge\; (1-e^{-1})\,p,
\end{equation}
and thus
\begin{equation}
\mathbb{E}[\,|U|\,] \;\ge\; (1-e^{-1})\,np
\;=\; (1-e^{-1})\,\alpha(n)\sqrt{n\log n}.
\end{equation}
Because $m\le \sqrt{n/\log n}$, we have $\mathbb{E}[\,|U|\,]=\omega(m\log n)$. Chernoff again gives
\begin{equation}
\mathbb{P}\Big(|U|<\tfrac12\,\mathbb{E}[\,|U|\,]\Big)
\;\le\; \exp\!\Big(-\tfrac{\mathbb{E}[\,|U|\,]}{8}\Big)
\;=\; \exp\!\big(-\omega(m\log n)\big).
\end{equation}

In either case, there exists an absolute constant \(c>0\) such that,
uniformly over all admissible \(m\) and all fixed choices of
\(x',x_1,\ldots,x_m\),
\begin{equation}\label{eq:uniform-expectation}
\mathbb E|U|
\ge c\alpha(n)m\log n,
\end{equation}
where
\[
p=\alpha(n)\sqrt{\frac{\log n}{n}}
\qquad\text{and}\qquad
\alpha(n)\to\infty.
\]
Consequently, by Chernoff's inequality, after decreasing \(c\) if necessary,
\begin{equation}\label{eq:uniform-expansion-tail}
\mathbb P\left(
|U|<\frac12\mathbb E|U|
\right)
\le
\exp\bigl(-c\alpha(n)m\log n\bigr).
\end{equation}
Moreover, on the complementary event,
\[
|U|
\ge \frac{c}{2}\alpha(n)m\log n
=\omega(m\log n).
\]

For each fixed \(m\), there are at most
\[
2n\binom{n-1}{m}\le 2n^{m+1}
\]
choices of the bipartition class and of the distinct vertices
\(x',x_1,\ldots,x_m\). Hence, by a union bound, the probability that the
desired conclusion fails for at least one admissible choice is at most
\begin{align*}
\sum_{m=1}^{\sqrt{n/\log n}}
2n^{m+1}\exp\bigl(-c\alpha(n)m\log n\bigr)
&\le
2\sum_{m\ge1}
\exp\bigl(-(c\alpha(n)-2)m\log n\bigr)\\
&=o(1),
\end{align*}
since \(\alpha(n)\to\infty\). This proves the lemma.
\end{proof}

\section{Proof of the Theorem}

\begin{prop}[No large crossings]\label{no_large_crossings}
Fix a constant \(c>0\) and let \(p=p(n)\) satisfy
\(p\gg \sqrt{\log n/n}\). Let \(G\sim G(n,n,p)\) be the random bipartite
graph with parts \(A\) and \(B\), where \(|A|=|B|=n\). Then, with high
probability, for every \(X\subseteq A\) and \(Y\subseteq B\) with
\[
|X|,|Y|\ge c\sqrt{n\log n},
\]
we have
\[
e_G(X,Y)\ge 1.
\]
\end{prop}

\begin{proof}
For subsets \(X\subseteq A\), \(Y\subseteq B\), let \(E(X,Y)\) denote the set
of edges with one endpoint in \(X\) and the other in \(Y\). Set
\(m=\lceil c\sqrt{n\log n}\rceil\). For any integers \(x,y\ge m\) and any
two sets \(X\subseteq A\), \(Y\subseteq B\) with \(|X|=x\), \(|Y|=y\), we have
\[
\mathbb P\big[E(X,Y)=\emptyset\big]=(1-p)^{xy}\le \exp(-pxy).
\]
A union bound, using symmetry between the two parts, gives
\begin{align}\label{eq:UB}
&\mathbb P\big( \text{there exist } X\subseteq A,\ Y\subseteq B
\text{ with } |X|,|Y|\ge m,\ E(X,Y)=\emptyset\big)\\
\le\;&
2\sum_{y=m}^{n}\sum_{x=y}^{n}
\binom{n}{x}\binom{n}{y}\exp(-pxy).
\nonumber
\end{align}
Below we estimate~\eqref{eq:UB}. For convenience, define
\begin{equation}
F(x,y)
:= \log\binom{n}{x}+\log\binom{n}{y}-p\,xy.
\end{equation}
Using \(\binom{n}{k}\le (en/k)^k\), we have
\begin{equation}\label{eq:F-bound}
F(x,y)
\le x\log\frac{en}{x}+y\log\frac{en}{y}-p\,xy.
\end{equation}

\medskip\noindent
\textbf{Step 1: For fixed \(y\), the exponent is maximized at \(x=y\).}
We consider
\begin{equation}
\phi_y(x):=x\log\frac{en}{x}-p\,xy \qquad (x\in[y,n]).
\end{equation}
Then
\begin{equation}
\phi_y'(x)=\log\frac{n}{x}-p\,y.
\end{equation}
Hence, for all \(x\ge y\),
\begin{equation}
\phi_y'(x)\le \log\frac{n}{y}-p\,y.
\end{equation}
If \(p\,y\ge \log(n/y)\), then \(\phi_y\) is nonincreasing on \([y,n]\). Therefore
\begin{equation}
\max_{x\in[y,n]}\left\{x\log\frac{en}{x}-p\,xy\right\}
= y\log\frac{en}{y}-p\,y^2.
\end{equation}
Consequently, under the condition \(p\,y\ge \log(n/y)\),
\begin{equation}\label{eq:max-at-diagonal}
\max_{x\ge y} F(x,y)
\le 2y\log\frac{en}{y}-p\,y^2.
\end{equation}

\medskip\noindent
\textbf{Step 2: For any fixed $c>0$, the choice $m=c\sqrt{n\log n}$ ensures $p\,y \ge \log\!\tfrac{n}{y}$ for all $y\ge m$.}
Write $p=\alpha(n)\sqrt{\tfrac{\log n}{n}}$ with $\alpha(n)\to\infty$ (since $p\gg \sqrt{\tfrac{\log n}{n}}$), and fix $c>0$.
Since $m=c\sqrt{n\log n}$,
for every $y\ge m$ we have
\begin{equation} 
p\,y \ \ge\ p\,m \ =\ \alpha(n)\,c\,\log n.
\end{equation}
Moreover, the function $y\mapsto \log\!\tfrac{n}{y}$ is decreasing on $[m,n]$, so
\begin{equation}
\log\!\tfrac{n}{y} \ \le\ \log\!\tfrac{n}{m}
= \log\!\Big(\frac{\sqrt{n}}{c\sqrt{\log n}}\Big)
= \tfrac12\log n - \tfrac12\log\log n - \log c
\ \le\ \tfrac12\log n,
\end{equation}
for all large $n$. 
Since $\alpha(n)\to\infty$, we have
$p\,y \ \ge\ \alpha(n)c\,\log n \ \ge\ \tfrac12\log n \ \ge\ \log\!\tfrac{n}{y}$
for all $y\ge m$ and all sufficiently large $n$. 
Therefore the condition
\(p\,y\ge \log(n/y)\) holds throughout the range \(y\ge m\), and hence
\eqref{eq:max-at-diagonal} applies.

\medskip
\noindent\textbf{Step 3. Summation along the diagonal \(x=y\).}
Combining \eqref{eq:UB} and \eqref{eq:max-at-diagonal}, and using that for
each fixed \(y\) there are at most \(n\) possible values of \(x\), we obtain
\begin{align}\label{eq:sum-diagonal}
&\mathbb P\big( \text{there exist } X\subseteq A,\ Y\subseteq B
\text{ with } |X|,|Y|\ge m,\ E(X,Y)=\emptyset\big)\\
\le\;&
2n\sum_{y=m}^{n}
\exp\!\Big(2y\log\tfrac{en}{y}-p\,y^2\Big).
\nonumber
\end{align}
For \(y\ge m\), write
\begin{equation}
2y\log\tfrac{en}{y}-p\,y^2
=
-y\left(p\,y-2\log\tfrac{en}{y}\right).
\end{equation}
By Step~2, \(p\,y\ge c\alpha(n)\log n\), where \(\alpha(n)\to\infty\). Moreover,
\begin{equation}
2\log\tfrac{en}{y}
\le
2\log\tfrac{en}{m}
=O(\log n).
\end{equation}
Hence, uniformly for all \(y\ge m\),
\begin{equation}
p\,y-2\log\tfrac{en}{y}
\ge \tfrac12 c\alpha(n)\log n
\end{equation}
for all sufficiently large \(n\). Therefore
\begin{equation}
2y\log\tfrac{en}{y}-p\,y^2
\le
-\tfrac12 c\alpha(n)y\log n
\le
-\tfrac12 c\alpha(n)m\log n
=
-\omega(m\log n).
\end{equation}
Consequently,
\begin{equation}
2n\sum_{y=m}^{n}
\exp\!\Big(2y\log\tfrac{en}{y}-p\,y^2\Big)
\le
2n^2\exp\big(-\omega(m\log n)\big)
=o(1).
\end{equation}
Thus, with high probability, there are no subsets
\(X\subseteq A\) and \(Y\subseteq B\) with \(|X|,|Y|\ge m\)
such that \(E(X,Y)=\emptyset\).
This completes the proof of the proposition.
\end{proof}

\begin{coro}[Large connected component in every dense pair]\label{big_component}
Fix any small absolute constant \(c_0>0\), and assume
\(p \gg \sqrt{\tfrac{\log n}{n}}\). Set
\[
s_0 := 100\,c_0\sqrt{n\log n}.
\]
With high probability, the following holds. If \(G\sim G(n,n,p)\) has
bipartition \(A\cup B\) with \(|A|=|B|=n\), then for every
\(X\subseteq A\) and \(Y\subseteq B\) with \(|X|,|Y|\ge s_0\), the induced
bipartite subgraph \(G[X,Y]\) contains a connected component \(C\) such that
\begin{equation}\label{eq:almost-all-in-one}
|(X\cup Y)\setminus C|
\le
c_0\sqrt{n\log n}.
\end{equation}
\end{coro}

\begin{proof}
Let $m:=\tfrac{c_0}{2}\sqrt{n\log n}$ and consider $G[X,Y]$.

\medskip
\noindent\textbf{Case 1.}
\emph{Every connected component \(D\) of \(G[X,Y]\) satisfies
\(|D\cap X|<m\) and \(|D\cap Y|<m\).}

Form a union \(S\) of components by adding components with nonempty
intersection with \(X\) one by one until \(|S\cap X|\ge m\), and stop as soon
as this happens. Since each added component contributes fewer than \(m\)
vertices of \(X\), we have
\[
m \le |S\cap X| < 2m.
\]
If \(|Y\setminus S|\ge m\), then there are no edges between \(S\cap X\) and
\(Y\setminus S\), because \(S\) is a union of connected components of
\(G[X,Y]\). This contradicts Proposition~\ref{no_large_crossings}. Hence
\[
|Y\setminus S|<m.
\]
Therefore
\[
|S\cap Y|=|Y|-|Y\setminus S|\ge s_0-m\ge m.
\]
Also,
\[
|X\setminus S|=|X|-|S\cap X|\ge s_0-2m\ge m.
\]
Since there are no edges between \(S\cap Y\) and \(X\setminus S\), this again
contradicts Proposition~\ref{no_large_crossings}.

\medskip
\noindent\textbf{Case 2.}
\emph{Some component \(C\) has a side of size at least \(m\).}

Let \(C\) be such a component, and assume \(|C\cap X|\ge m\); the case
\(|C\cap Y|\ge m\) is analogous. If \(|Y\setminus C|\ge m\), take any
\(m\)-set \(Y_0\subseteq Y\setminus C\). By
Proposition~\ref{no_large_crossings} applied to \((C\cap X,Y_0)\), there
exists an edge between \(C\cap X\) and \(Y_0\), contradicting the fact that
\(C\) is a connected component of \(G[X,Y]\). Therefore
\[
|Y\setminus C|<m.
\]
Hence
\[
|C\cap Y|=|Y|-|Y\setminus C|\ge s_0-m\ge m.
\]
Now let \(X_0:=X\setminus C\). If \(|X_0|\ge m\), then applying
Proposition~\ref{no_large_crossings} to \((X_0,C\cap Y)\) gives an edge
joining \(X_0\) to \(C\), again contradicting that \(C\) is a connected
component. Thus \(|X_0|<m\), and we conclude that
\[
|(X\cup Y)\setminus C|
=
|X\setminus C|+|Y\setminus C|
<2m
=
c_0\sqrt{n\log n}.
\]
This proves the corollary.
\end{proof}

We are now ready to prove the main theorem.

\begin{proof}[Proof of Theorem~\ref{thm:three-component-bipartite}]
Let \(p \gg \sqrt{\tfrac{\log n}{n}}\), and let \(G \sim G(n,n,p)\) with
bipartition \((A,B)\), where \(|A|=|B|=n\). Fix a small absolute constant
\(c_0>0\), and set
\[
s_0=s_0(n):=100c_0\sqrt{n\log n}.
\]

We work on the high-probability event on which Lemma~\ref{degree_codegree},
Lemma~\ref{logn_expanding}, and Corollary~\ref{big_component} all hold.
Thus every vertex \(v\) of \(G\) satisfies
\begin{equation}
    \deg_G(v)\ge \frac12 pn \gg \sqrt{n\log n},
\end{equation}
and every pair of distinct vertices \(u,u'\) in the same part satisfies
\begin{equation}
|N(u)\cap N(u')|\ge \frac12 p^2n\gg \log n.
\end{equation}
Moreover, for all subsets \(X\subseteq A\) and \(Y\subseteq B\) with
\(|X|,|Y|\ge s_0\), the induced bipartite subgraph \(G[X,Y]\) contains a
connected component \(C\) such that
\begin{equation}
|(X\cup Y)\setminus C|\le c_0\sqrt{n\log n}.
\end{equation}

It remains to prove that every red--blue edge-colouring of this fixed graph
\(G\) admits a cover of \(V(G)\) by at most three monochromatic connected
components. Fix an arbitrary red--blue edge-colouring of \(E(G)\).
Among all monochromatic connected components, choose one whose larger side is
as large as possible. Renaming colours if necessary, let this component be
blue, and denote it by \(C\), with sides \(A_1\subseteq A\) and
\(B_1\subseteq B\), where
\(|B_1|\ge |A_1|.
\)
By Lemma~\ref{degree_codegree}, for any vertex \(v\), one of its two
colour-neighbourhoods has size at least
\( \frac12\deg_G(v)\ge \frac14pn\gg s_0. \)
Hence some monochromatic connected component has one side of size at least
\(\frac14pn\gg s_0\). By the choice of \(C\), we may assume \(|B_1|\ge s_0.
\)
Denote
\[
A_1^{c}:=A\setminus A_1,
\qquad
B_1^{c}:=B\setminus B_1.
\]

We first show that the larger side of \(C\) must be close to \(n/2\). Suppose
for contradiction that
\begin{equation}\label{eq:B1-too-small}
|B_1|\le \frac n2-\frac{s_0}{100}.
\end{equation}
Since \(|A_1|\le |B_1|\), we have
\[
|A_1^c|=n-|A_1|
\ge n-|B_1|
\ge \frac n2+\frac{s_0}{100}.
\]
By maximality of the blue component \(C\), there are no blue edges between
\(B_1\) and \(A_1^c\). Therefore every present edge in \(G[B_1,A_1^c]\) is red.
Since \(|B_1|\ge s_0\) and \(|A_1^c|\ge s_0\), Corollary~\ref{big_component}
applied to \(G[B_1,A_1^c]\) gives a red component \(R\) such that
\[
|(B_1\cup A_1^c)\setminus R|
\le c_0\sqrt{n\log n}
=
\frac{s_0}{100}.
\]
Consequently,
\[
|A\cap R|
\ge |A_1^c|-\frac{s_0}{100}
\ge
\left(\frac n2+\frac{s_0}{100}\right)-\frac{s_0}{100}
=
\frac n2.
\]
Thus \(R\) is a monochromatic component whose larger side has size at least
\(n/2\), contradicting the choice of \(C\). 
Therefore the larger side $B_1$ satisfies 
\begin{equation}\label{eq:B1-large}
|B_1|> \frac n2-\frac{s_0}{100}.
\end{equation}

We next show that the smaller side of our chosen component may also be assumed
to be close to \(n/2\). Suppose that
\[
|A_1|<\frac n2-s_0.
\]
Then
\[
|A_1^c|>\frac n2+s_0.
\]
By maximality of the blue component \(C\), there are no blue edges between
\(B_1\) and \(A_1^c\), so every present edge in \(G[B_1,A_1^c]\) is red.
Using \eqref{eq:B1-large} and applying Corollary~\ref{big_component} to
\(G[B_1,A_1^c]\), we obtain a red component \(R\) such that
\[
|(B_1\cup A_1^c)\setminus R|
\le c_0\sqrt{n\log n}
=
\frac{s_0}{100}.
\]
Therefore
\[
|R\cap A_1^c|
\ge |A_1^c|-\frac{s_0}{100}
>
\frac n2+s_0-\frac{s_0}{100}
>
\frac n2-s_0,
\]
and also
\[
|R\cap B_1|
\ge |B_1|-\frac{s_0}{100}
>
\left(\frac n2-\frac{s_0}{100}\right)-\frac{s_0}{100}
=
\frac n2-\frac{s_0}{50}
>
\frac n2-s_0.
\]
Thus \(R\) is a red connected subgraph whose two sides both have size at
least \(n/2-s_0\). Let \(R^\ast\) be the maximal red component of \(G\)
containing \(R\). Then \(R^\ast\) also has both sides of size at least
\(n/2-s_0\).

Consequently, replacing \(C\) by \(R^\ast\) if necessary, and renaming colours
and interchanging the two bipartition classes if necessary, we may assume that
\(C\) is a blue component with sides \(A_1\subseteq A\) and \(B_1\subseteq B\),
where
\begin{equation}\label{A_1B_1}
|B_1|\ge |A_1|\ge \frac n2-s_0\gg s_0.
\end{equation}

The remainder of the proof consists of a rather long case analysis. 
Although the discussion is somewhat technical, each case follows the same guiding idea: 
starting from the large monochromatic component $C$, we analyze how the remaining vertices 
can be absorbed into at most two further monochromatic components. We introduce the threshold $2s_0$ to separate the cases; this choice is purely technical.\\

\noindent\textbf{Case~1.} 
$|A_1^c| < 2s_0$ and $|B_1^c| < 2s_0$.

\medskip
\noindent\textbf{Case~(1.1).
One of the remainder sets, say $A_1^c$, is contained in a single maximal red component.}
\smallskip
Let \(R\) denote the maximal red component containing \(A_1^c\), and set
\[
B' := B_1^c\setminus R.
\]
We claim that \(B'\) is contained in a single monochromatic component. If
\(|B'|\le 1\), this is immediate, so assume \(|B'|\ge 2\).

By the codegree estimate from Lemma~\ref{degree_codegree}, for any two
distinct vertices \(y_1,y_2\in B\) we have
\[
|N(y_1)\cap N(y_2)|\ge \frac12p^2n\gg \log n.
\]
Hence, for every pair \(y_1,y_2\in B'\), the common neighbourhood
\(N(y_1)\cap N(y_2)\) is nonempty and lies in \(A=A_1\cup A_1^c\).

If \((N(y_1)\cap N(y_2))\cap A_1\neq\varnothing\), choose
\(x\in A_1\cap N(y_1)\cap N(y_2)\). Since \(y_1,y_2\in B_1^c\), maximality of
the blue component \(C\) implies that the edges \(xy_1\) and \(xy_2\) are red.
Thus \(y_1\) and \(y_2\) are red-connected.

If \((N(y_1)\cap N(y_2))\cap A_1^c\neq\varnothing\), choose
\(x\in A_1^c\cap N(y_1)\cap N(y_2)\). Since \(A_1^c\subseteq R\) and
\(y_1,y_2\notin R\), the edges \(xy_1\) and \(xy_2\) cannot be red. Hence they
are blue, and \(y_1\) and \(y_2\) are blue-connected.

To formalize this, consider the auxiliary \(2\)-edge-coloured complete graph
\(H\) on vertex set \(B'\), where two vertices are joined by a red edge if
they are red-connected in \(G\), and by a blue edge if they are blue-connected
in \(G\); if both hold, choose either colour arbitrarily. By
Lemma~\ref{lem:complete-graph-two-colour}, \(H\) contains a monochromatic
connected component spanning all of \(B'\). Therefore, in the original graph
\(G\), the set \(B'\) is contained in a single monochromatic connected
component.

Consequently, the entire vertex set of \(G\) is covered by at most three
monochromatic connected components: the blue component \(C\), the red
component \(R\), and the monochromatic component containing \(B'\).

\noindent\textbf{Case (1.2). Both \(A_1^c\) and \(B_1^c\) are not contained in a single red component.}
Recall that \eqref{A_1B_1} implies
\begin{equation}\label{A1cB1c}
|B_1^c|\le |A_1^c|.
\end{equation}

Since \(A_1^c\) is not contained in a single red component, we may choose
a proper union \(S\) of red components meeting \(A_1^c\) such that
\[
\frac12|A_1^c|
\le |S\cap A_1^c|
<|A_1^c|.
\]
Indeed, if one red component contains at least half of \(A_1^c\), choose it;
otherwise, add red components one at a time until their union first contains
at least half of \(A_1^c\).

By maximality of the blue component \(C\), all present edges between
\(A_1^c\) and \(B_1\) are red. Hence if a vertex \(b\in B_1\) belonged to
\(N(x_i)\cap N(x'_1)\), then \(b\) would form a red path
\[
x_i-b-x'_1,
\]
contradicting the fact that \(x_i\) and \(x'_1\) are not red-connected.
Therefore, for every \(i\),
\begin{equation}\label{eq:common-neigh-in-B1c}
N(x_i)\cap N(x'_1)\subseteq B_1^c.
\end{equation}

If \(k\le \sqrt{n/\log n}\), apply Lemma~\ref{logn_expanding} with \(m=k\)
to the vertices \(x'_1,x_1,\ldots,x_k\). Using
\eqref{eq:common-neigh-in-B1c}, we obtain
\begin{equation}
|B_1^c|
\ge
\left|\bigcup_{i=1}^k \bigl(N(x_i)\cap N(x'_1)\bigr)\right|
=
\omega(k\log n).
\end{equation}
Since \(k\ge |A_1^c|/2\), the right-hand side is larger than
\(|A_1^c|\) for all large \(n\), contradicting \eqref{A1cB1c}.

If \(k>k_0:=\sqrt{n/\log n}\), apply Lemma~\ref{logn_expanding} with
\(m=k_0\) to \(x'_1,x_1,\ldots,x_{k_0}\). Again using
\eqref{eq:common-neigh-in-B1c}, we get
\begin{equation}
|B_1^c|
\ge
\left|\bigcup_{i=1}^{k_0} \bigl(N(x_i)\cap N(x'_1)\bigr)\right|
=
\omega(k_0\log n)
=
\omega(\sqrt{n\log n}).
\end{equation}
This contradicts Case~1, where \(|B_1^c|<2s_0=O(\sqrt{n\log n})\).

Therefore Case~(1.2) cannot occur. We refer to the above reasoning as the common-neighbourhood expansion argument; in later applications, the argument will be iterated alternately between the two bipartition classes.

\noindent\textbf{Case 2.}
\(|A_1^c|\ge 2s_0\) and \(|B_1^c|<2s_0\).

\smallskip
In this case, consider the bipartite subgraph \(G[A_1^c,B_1]\). Since the
blue component \(C\) with sides \(A_1\) and \(B_1\) is maximal, there are no
blue edges between \(A_1^c\) and \(B_1\). Thus every present edge of
\(G[A_1^c,B_1]\) is red. Since \(|A_1^c|\ge 2s_0\) and
\[
|B_1|\ge \frac n2-s_0\gg s_0,
\]
we may apply Corollary~\ref{big_component}. It yields a red connected
component \(R_0\) such that
\begin{equation}\label{eq:case2-R}
\bigl|(A_1^c\cup B_1)\setminus R_0\bigr|
\le c_0\sqrt{n\log n}.
\end{equation}

Let \(R\) be the maximal red component of \(G\) containing \(R_0\). We claim
that \(A_1^c\subseteq R\). Indeed, take any vertex
\(x\in A_1^c\setminus R_0\), if such a vertex exists. By the degree estimate,
\[
|N(x)|\gg \sqrt{n\log n}.
\]
On the other hand,
\begin{equation}\label{eq:case2-BminusR}
|B\setminus R_0|
=
|B_1^c|+|B_1\setminus R_0|
\le
2s_0+c_0\sqrt{n\log n}
=
O(\sqrt{n\log n}).
\end{equation}
Therefore \(N(x)\) must intersect \(B_1\cap R_0\). Since every edge between
\(A_1^c\) and \(B_1\) is red, \(x\) is joined by a red edge to \(R_0\), and
hence \(x\in R\). Thus \(A_1^c\subseteq R\).

At this point we are in the situation of Case~(1.1): one of the two remainder
sets, namely \(A_1^c\), is contained in a single maximal red component. Hence
the same argument as in Case~(1.1) shows that \(V(G)\) is covered by at most
three monochromatic connected components: the blue component \(C\), the red
component \(R\), and one additional monochromatic component covering
\(B_1^c\setminus R\).

\noindent\textbf{Case 3.}
\(|A_1^c|\ge |B_1^c|\ge 2s_0\).

\smallskip
First consider the bipartite subgraph \(G[A_1^c,B_1]\). By maximality of the
blue component \(C\), there are no blue edges between \(A_1^c\) and \(B_1\);
hence every present edge of \(G[A_1^c,B_1]\) is red. Since
\(|A_1^c|\ge 2s_0\) and \(|B_1|\ge n/2-s_0\gg s_0\), Corollary~\ref{big_component}
gives a red connected component \(R_0\) in \(G[A_1^c,B_1]\) such that
\[
\bigl|(A_1^c\cup B_1)\setminus R_0\bigr|
\le c_0\sqrt{n\log n}.
\]
Let \(R\) be the maximal red component of \(G\) containing \(R_0\). Define
\[
A_2:=A_1^c\cap R,\qquad
Q_1:=B_1\setminus R,\qquad
X:=A_1^c\setminus R.
\]
Equivalently, \(X=A\setminus(A_1\cup A_2)\). Since \(R\supseteq R_0\), we have
\[
|Q_1|,|X|\le c_0\sqrt{n\log n}.
\]
If \(X=\emptyset\), then \(A_1^c\subseteq R\), and we are in the situation of
Case~(1.1). Hence \(V(G)\) is covered by at most three monochromatic connected
components. Thus we may assume \(X\neq\emptyset\).

Next consider the bipartite subgraph \(G[A_1,B_1^c]\). Again by maximality of
the blue component \(C\), there are no blue edges between \(A_1\) and \(B_1^c\);
hence every present edge of \(G[A_1,B_1^c]\) is red. Since
\(|A_1|\ge n/2-s_0\gg s_0\) and \(|B_1^c|\ge 2s_0\), Corollary~\ref{big_component}
gives a red connected component \(R_0'\) in \(G[A_1,B_1^c]\) such that
\[
\bigl|(A_1\cup B_1^c)\setminus R_0'\bigr|
\le c_0\sqrt{n\log n}.
\]
Let \(R'\) be the maximal red component of \(G\) containing \(R_0'\). Define
\[
B_2:=B_1^c\cap R',\qquad
P_1:=A_1\setminus R',\qquad
Y:=B_1^c\setminus R'.
\]
Equivalently, \(Y=B\setminus(B_1\cup B_2)\). Since \(R'\supseteq R_0'\), we have
\[
|P_1|,|Y|\le c_0\sqrt{n\log n}.
\]
If \(Y=\emptyset\), then \(B_1^c\subseteq R'\), and the symmetric version of
Case~(1.1) applies. Hence \(V(G)\) is covered by at most three monochromatic
connected components. Thus we may assume \(Y\neq\emptyset\).

If \(R\) and \(R'\) are the same red component, then this red component covers
all vertices except possibly those in
\[
P_1\cup X\cup Q_1\cup Y.
\]
Indeed, by the definitions above,
\[
A\setminus (R\cup R')\subseteq P_1\cup X,
\qquad
B\setminus (R\cup R')\subseteq Q_1\cup Y.
\]
Since
\[
|P_1|,|X|,|Q_1|,|Y|\le c_0\sqrt{n\log n},
\]
each side of the complement of this red component has size at most
\(2c_0\sqrt{n\log n}<2s_0\). Hence, applying the same argument as in
Case~1, with this red component playing the role of the initial large
component, we obtain a cover of \(V(G)\) by at most three monochromatic
connected components.

Therefore, from now on we may assume that \(R\) and \(R'\) are distinct red
components. In particular, there is no red edge between \(R\) and \(R'\).

Finally, consider the bipartite subgraph \(G[A_2,B_2]\). Since \(A_2\subseteq R\)
and \(B_2\subseteq R'\), and \(R\) and \(R'\) are distinct red components, there
are no red edges between \(A_2\) and \(B_2\). Hence every present edge of
\(G[A_2,B_2]\) is blue.

Moreover,
\[
|A_2|\ge |A_1^c|-|X|\ge 2s_0-c_0\sqrt{n\log n}\ge s_0,
\]
and similarly \(|B_2|\ge s_0\). Thus, by Corollary~\ref{big_component}, there
exists a blue connected component \(S\) in \(G[A_2,B_2]\) such that
\[
|(A_2\cup B_2)\setminus S|\le c_0\sqrt{n\log n}.
\]
Enlarge \(S\) to a maximal blue connected component of \(G\), and denote it by
\(R''\). Define
\[
P_2:=A_2\setminus R'',
\qquad
Q_2:=B_2\setminus R''.
\]
Since \(R''\supseteq S\), we have
\[
|P_2|,|Q_2|\le c_0\sqrt{n\log n}.
\]

By maximality of \(R''\), there are no blue edges from \(P_2\) to
\(B_2\setminus Q_2\), and no blue edges from \(Q_2\) to \(A_2\setminus P_2\).
Since every present edge between \(A_2\) and \(B_2\) is blue, it follows that
there are in fact no edges from \(P_2\) to \(B_2\setminus Q_2\), and no edges
from \(Q_2\) to \(A_2\setminus P_2\)\\

\noindent\textbf{Claim.} We have
\[
X\subseteq R'\cup R''
\qquad\text{and}\qquad
Y\subseteq R\cup R''.
\]

\begin{proof}
Fix \(x\in X\). Since \(X=A_1^c\setminus R\), we have
\[
N(x)\cap (B_1\cap R)=\emptyset.
\]
Indeed, every edge between \(A_1^c\) and \(B_1\) is red, so any neighbour of
\(x\) in \(B_1\cap R\) would put \(x\) in the red component \(R\).

Using the disjoint partition
\[
B=(B_1\cap R)\sqcup Q_1\sqcup Q_2\sqcup (B_2\cap R'')\sqcup Y,
\]
where \(Q_1=B_1\setminus R\) and \(Q_2=B_2\setminus R''\), we obtain
\[
N(x)\subseteq Q_1\cup Q_2\cup (B_2\cap R'')\cup Y.
\]
Since
\[
|Q_1|,|Q_2|,|Y|\le c_0\sqrt{n\log n}
\]
while
\[
|N(x)|\gg \sqrt{n\log n},
\]
it follows that
\[
N(x)\cap (B_2\cap R'')\neq\emptyset.
\]
Choose \(v\in N(x)\cap (B_2\cap R'')\). If \(xv\) is red, then
\(x\in R'\), because \(v\in B_2\subseteq R'\). If \(xv\) is blue, then
\(x\in R''\), because \(v\in R''\) and \(R''\) is a maximal blue component.
Thus \(x\in R'\cup R''\).

\smallskip
The argument for \(Y\) is symmetric. Fix \(y\in Y\). Since
\(Y=B_1^c\setminus R'\), we have
\[
N(y)\cap (A_1\cap R')=\emptyset.
\]
Using the disjoint partition
\[
A=(A_1\cap R')\sqcup P_1\sqcup P_2\sqcup (A_2\cap R'')\sqcup X,
\]
where \(P_1=A_1\setminus R'\) and \(P_2=A_2\setminus R''\), we obtain
\[
N(y)\subseteq P_1\cup P_2\cup (A_2\cap R'')\cup X.
\]
Since
\[
|P_1|,|P_2|,|X|\le c_0\sqrt{n\log n}
\]
while
\[
|N(y)|\gg \sqrt{n\log n},
\]
we obtain
\[
N(y)\cap (A_2\cap R'')\neq\emptyset.
\]
Choose \(u\in N(y)\cap (A_2\cap R'')\). If \(yu\) is red, then
\(y\in R\), because \(u\in A_2\subseteq R\). If \(yu\) is blue, then
\(y\in R''\), because \(u\in R''\) and \(R''\) is a maximal blue component.
Thus \(y\in R\cup R''\).
\end{proof}

\noindent\textbf{Case (3.1). Both \(P_1\) and \(Q_1\) are empty.}
In this case, every vertex of \(A_1\) lies in \(R'\), and every vertex of
\(B_1\) lies in \(R\). Hence the initial blue component \(C\) is contained in
\(R\cup R'\).

Moreover, \(A_2\subseteq R\) and \(B_2\subseteq R'\), while by the claim we have
\[
X\subseteq R'\cup R''
\qquad\text{and}\qquad
Y\subseteq R\cup R''.
\]
Therefore the three monochromatic components \(R\), \(R'\), and \(R''\) cover
all vertices of \(G\), and the desired conclusion holds.\\

\noindent\textbf{Case (3.2). Both \(P_1\) and \(Q_1\) are non-empty.}
Let us consider the two sets \(P_1\cup X\) and \(Q_1\cup Y\). Without loss of
generality, suppose that
\[
|P_1\cup X|=:m\ge |Q_1\cup Y|.
\]
Take \(p\in P_1\) and \(x\in X\). We claim that
\[
N(p)\cap N(x)\subseteq Q_1\cup Y.
\]
Indeed, \(x\) has no neighbour in \(B_1\cap R\), because every edge between
\(A_1^c\) and \(B_1\) is red, and such a neighbour would put \(x\) in \(R\).
Also, \(p\) has no neighbour in \(B_2\), because \(B_2\subseteq R'\), every
edge between \(A_1\) and \(B_1^c\) is red, and such a neighbour would put
\(p\) in \(R'\). Since
\[
B=(B_1\cap R)\sqcup Q_1\sqcup B_2\sqcup Y,
\]
we obtain \(N(p)\cap N(x)\subseteq Q_1\cup Y\).

Suppose first that \(|X|\ge m/2\). If
\[
\frac m2\le \sqrt{\frac n{\log n}},
\]
choose distinct vertices \(x_1,\ldots,x_{m/2}\in X\). Applying
Lemma~\ref{logn_expanding} to \(p,x_1,\ldots,x_{m/2}\), and using
\(N(p)\cap N(x_j)\subseteq Q_1\cup Y\), we get
\[
|Q_1\cup Y|
\ge
\left|\bigcup_{j=1}^{m/2}\bigl(N(p)\cap N(x_j)\bigr)\right|
=
\omega(m\log n),
\]
contradicting \(|Q_1\cup Y|\le m\).

If instead
\[
\frac m2> \sqrt{\frac n{\log n}},
\]
let \(k_0:=\sqrt{n/\log n}\), and choose distinct
\(x_1,\ldots,x_{k_0}\in X\). Applying Lemma~\ref{logn_expanding} to
\(p,x_1,\ldots,x_{k_0}\), we obtain
\[
|Q_1\cup Y|
\ge
\left|\bigcup_{j=1}^{k_0}\bigl(N(p)\cap N(x_j)\bigr)\right|
=
\omega(k_0\log n)
=
\omega(\sqrt{n\log n}).
\]
This contradicts
\[
|Q_1\cup Y|\le |Q_1|+|Y|\le 2c_0\sqrt{n\log n}.
\]

Therefore \(|X|<m/2\), and hence \(|P_1|\ge m/2\). Fix any \(x\in X\). The
same argument with the roles of \(P_1\) and \(X\) interchanged gives a
contradiction. Indeed, if \(m/2\le \sqrt{n/\log n}\), choose
\(p_1,\ldots,p_{m/2}\in P_1\) and apply Lemma~\ref{logn_expanding} to
\(x,p_1,\ldots,p_{m/2}\). If \(m/2>\sqrt{n/\log n}\), choose
\(p_1,\ldots,p_{k_0}\in P_1\). In both cases, using
\(N(p_j)\cap N(x)\subseteq Q_1\cup Y\), we obtain the same contradiction.

Hence Case~(3.2) cannot occur.\\

\noindent\textbf{Case (3.3). Exactly one of the two sets \(P_1\) and
\(Q_1\) is empty.}
By symmetry between the two bipartition classes, we may assume that
\(P_1\neq\emptyset\) and \(Q_1=\emptyset\).

We first record a consequence of the same argument used in Case~(3.2). For any
\(p\in P_1\) and \(x\in X\), we have
\[
N(p)\cap N(x)\subseteq Y.
\]
Indeed, since \(Q_1=\emptyset\), we have \(B_1\subseteq R\), and \(x\in X\)
has no neighbour in \(B_1\). Also, \(p\in P_1\) has no neighbour in \(B_2\),
as otherwise \(p\) would be red-connected to \(R'\). Hence every common
neighbour of \(p\) and \(x\) must lie in \(Y\).

Let
\[
M:=|P_1\cup X|.
\]
If \(|X|\ge M/2\), fix \(p\in P_1\) and apply Lemma~\ref{logn_expanding} to
\(p\) together with either \(M/2\) vertices of \(X\), if
\(M/2\le \sqrt{n/\log n}\), or with \(\sqrt{n/\log n}\) vertices of \(X\)
otherwise. Using \(N(p)\cap N(x)\subseteq Y\), we obtain
\[
|Y|=\omega(M).
\]
If instead \(|X|<M/2\), then \(|P_1|\ge M/2\), and the same argument with a
fixed \(x\in X\) and many vertices of \(P_1\) again gives
\[
|Y|=\omega(M).
\]
Thus
\begin{equation}\label{eq:Y-large-case33}
|Y|=\omega(|P_1\cup X|).
\end{equation}
We keep this fact in mind for the remaining argument.\\

We further split into subcases below. One note is that, if the relevant set \(T\) has size larger than
\(\sqrt{n/\log n}\), we apply Lemma~\ref{logn_expanding} to an arbitrary
\(\sqrt{n/\log n}\)-subset of \(T\). Since
\(|T|=O(\sqrt{n\log n})\), the resulting bound
\(\omega(\sqrt{n\log n})\) is still \(\omega(|T|)\).

\noindent\textbf{Case (3.3.1). \(Q_2 \neq \emptyset\).}

Set
\[
U:=P_2\cup X,
\qquad
V:=Q_2\cup Y.
\]
We first observe that, for every \(q\in Q_2\) and every \(y\in Y\),
\[
N(q)\cap N(y)\subseteq U.
\]
Indeed, a common neighbour cannot lie in \(A_1\): since \(q,y\in B_1^c\), all
edges from \(A_1\) to \(B_1^c\) are red, and such a common neighbour would
red-connect \(y\) to \(q\in B_2\subseteq R'\), contradicting \(y\notin R'\).
Also, a common neighbour cannot lie in \(A_2\cap R''\), because
\(q\in Q_2=B_2\setminus R''\) has no edge to \(A_2\cap R''\). Hence all common
neighbours lie in
\[
(A_2\setminus R'')\cup X=P_2\cup X=U.
\]

Similarly, for every \(p\in P_2\) and every \(x\in X\), we have \(
N(p)\cap N(x)\subseteq V.
\)

Now fix \(q\in Q_2\) and apply the alternating common-neighbourhood expansion argument to the pairs
\((q,y)\), \(y\in Y\). Since \(N(q)\cap N(y)\subseteq U\), Lemma~\ref{logn_expanding}
implies
\[
|U|=\omega(|Y|).
\]
By \eqref{eq:Y-large-case33}, we have \(|Y|=\omega(|P_1\cup X|)\), and in
particular \(|X|=o(|Y|)\). Therefore
\[
|P_2|=\omega(|Y|).
\]
In particular, \(P_2\neq\emptyset\).

Next fix \(x\in X\) and apply the same argument to the pairs \((p,x)\),
\(p\in P_2\). Since \(N(p)\cap N(x)\subseteq V\), we get
\[
|V|=\omega(|P_2|).
\]
As \(|P_2|=\omega(|Y|)\), this forces
\[
|Q_2|=\omega(|P_2|).
\]

Finally, fix \(y\in Y\) and apply the first containment again to the pairs
\((q,y)\), \(q\in Q_2\). We obtain
\[
|U|=\omega(|Q_2|).
\]
But \(U=P_2\cup X\), and since \(|X|=o(|Y|)\) while \(|P_2|=\omega(|Y|)\), we
have
\[
|U|=(1+o(1))|P_2|.
\]
This contradicts \(|Q_2|=\omega(|P_2|)\). Hence Case~(3.3.1) cannot occur.\\

\noindent\textbf{Case (3.3.2). \(Q_2=\emptyset\).}
First suppose that \(X\cap R'=\emptyset\). By the claim, we have
\[
X\subseteq R'\cup R''.
\]
Hence \(X\subseteq R''\). Also, again by the claim,
\[
Y\subseteq R\cup R''.
\]
Since \(Q_2=\emptyset\), we have \(B_2\subseteq R''\). Moreover,
\(A_1\cup B_1\subseteq C\) and \(A_2\subseteq R\). Therefore the three
monochromatic components \(C\), \(R\), and \(R''\) cover all vertices of \(G\).

Now suppose that \(X\cap R'\neq\emptyset\).

If \(P_1\subseteq R\), then, since \(P_1=A_1\setminus R'\),
\[
A_1=(A_1\cap R')\cup P_1\subseteq R'\cup R.
\]
Moreover, \(Q_1=\emptyset\) implies \(B_1\subseteq R\). Hence
\[
C\subseteq R\cup R'.
\]
Together with \(A_2\subseteq R\), \(B_2\subseteq R'\), and the claim
\[
X\subseteq R'\cup R'',
\qquad
Y\subseteq R\cup R'',
\]
it follows that \(R,R'\), and \(R''\) cover all vertices of \(G\).

We may therefore assume that there exists
\[
p\in P_1\setminus R.
\]
Since \(P_1=A_1\setminus R'\), we also have \(p\notin R'\).

We claim that \(X\subseteq R''\). Let \(x\in X\cap R'\). Since \(p\) and \(x\)
are distinct vertices in the same part of the bipartition, the codegree
estimate gives
\[
N(p)\cap N(x)\neq\emptyset.
\]
As observed in Case~(3.3), we have
\[
N(p)\cap N(x)\subseteq Y.
\]
Choose \(y\in N(p)\cap N(x)\). Since \(p\in A_1\) and
\(y\in Y\subseteq B_1^c\), the edge \(py\) is red. If \(y\in R\), then the
red edge \(py\) would imply \(p\in R\), a contradiction. Hence \(y\notin R\).
By the claim \(Y\subseteq R\cup R''\), it follows that
\[
y\in R''.
\]

If the edge \(xy\) were red, then \(x\in R'\) would imply \(y\in R'\).
The red edge \(py\) would then imply \(p\in R'\), contradicting
\(p\notin R'\). Therefore \(xy\) is blue. Since \(y\in R''\) and \(R''\)
is a maximal blue component, we obtain \(x\in R''\). Thus
\[
X\cap R'\subseteq R''.
\]
Together with the claim \(X\subseteq R'\cup R''\), this yields
\[
X\subseteq R''.
\]

Finally, \(Q_2=\emptyset\) implies \(B_2\subseteq R''\). Also,
\[
A_1\cup B_1\subseteq C,\qquad A_2\subseteq R,
\qquad Y\subseteq R\cup R''.
\]
Therefore the three monochromatic components \(C\), \(R\), and \(R''\)
cover all vertices of \(G\).
\end{proof}

\bibliographystyle{plain}
\addcontentsline{toc}{chapter}{Bibliography}
\bibliography{Gnpcombined}
\end{document}